\documentclass[12pt]{article}
\usepackage{amsmath,amssymb,amsthm}

\newcommand{\fgs}{{\mbox{\rm f.$G$-s.}}}

\def\1{\underline{1}}

\def\P{\mathbb P}

\def\Z{{\mathbb Z}}
\def\Q{{\mathbb Q}}
\def\C{{\mathbb C}}

\def\DD{{\cal D}}

\def\gF{{\mathfrak F}}
\def\gH{{\mathfrak H}}

\newtheorem*{theorem*}{Theorem}

\newtheorem{lemma}{Lemma}
\newtheorem{proposition}{Proposition}

\newenvironment{definition}
{\smallskip\noindent{\bf Definition\/}:}{\smallskip\par}

\newenvironment{example}
{\smallskip\noindent{\bf Example\/}.}{\smallskip\par}
\newenvironment{examples}
{\smallskip\noindent{\bf Examples\/}.}{\smallskip\par}
\newenvironment{remark}
{\smallskip\noindent{\bf Remark\/}.}{\smallskip\par}

\title{On an equivariant version of the zeta function of a transformation
}
\author{S.M.~Gusein-Zade \thanks{Partially supported by
the Russian government grant 11.G34.31.0005,
RFBR--13-01-00755,
NSh--4850.2012.1 and Simons-IUM fellowship.
Address: Moscow State University, Faculty
of Mathematics and Mechanics, GSP-1, Moscow, 119991, Russia. E-mail:
sabir\symbol{'100}mccme.ru} \and I.~Luengo \thanks{The last two authors are partially
supported by the grant MTM2010-21740-C02-01. Address: University
Complutense de Madrid, Dept. of Algebra, Madrid, 28040, Spain.
E-mail: iluengo\symbol{'100}mat.ucm.es} \and
A.~Melle--Hern\'andez \thanks{Address: 
ICMAT (CSIC-UAM-UC3M-UCM). Dept.\ of Algebra, 
Facultad de Ciencias Matem\'aticas, Universidad Complutense de Madrid, 
28040, Madrid, Spain.
E-mail: amelle\symbol{'100}mat.ucm.es}}
\date{}
\begin{document}
\def\eps{\varepsilon}

\maketitle

\begin{abstract}
Earlier the authors offered an equivariant version of the classical monodromy zeta function of a
$G$-invariant function germ with a finite group $G$ as a power series with the coefficients from
the Burnside ring $K_0(\fgs)$ of the group $G$ tensored by
the field of rational numbers.
One of the main ingredients of the definition was the definition of the equivariant Lefschetz number
of a $G$-equivariant transformation given by W.~L\"uck and J.~Rosenberg.
Here we offer another approach to a definition of the equivariant Lefschetz number of a transformation and
describe the corresponding notions of the equivariant zeta function. This zeta-function
is a power series with the coefficients from the ring $K_0(\fgs)$.
We give an A'Campo type formula for the equivariant monodromy zeta function
of a function germ in terms of a resolution.
Finally we discuss orbifold versions of the Lefschetz number and of the monodromy zeta function
corresponding to the two equivariant ones.
\end{abstract}

Many topological invariants have equivariant versions for spaces with actions of a group $G$,
say, a finite one. For example, in \cite{Verdier}, an {equivariant} version of the Euler
characteristic is an element of the Grothendieck ring of $\Z[G]$- or $\Q[G]$-modules.
In \cite[Section~5.4]{TtD} it is defined as an element of the Burnside ring of the group $G$
(that is of the Grothendieck ring $K_0(\fgs)$ of finite $G$-sets). Applying these concepts to
the Milnor fibre, one gets an equivariant version of the Milnor number of a $G$-invariant
function-germ. For example, in \cite{Wall} it is an element of the ring of virtual
representations of the group $G$.

An important invariant of a germ of a holomorphic function (on $(\C^n,0)$ or on a germ of a
complex analytic variety) is its monodromy zeta function. It is defined as the zeta function
of the classical monodromy transformation on the Milnor fibre. The monodromy zeta function
is connected with a number of other invariants, topological and analytic ones. For example,
in \cite{CDG1}, it was shown that, for an irreducible plane curve singularity, the monodromy
zeta function of the corresponding function-germ coincides with the Poincar\'e series of
the natural fitration on the local ring defined by the 
curve valuation. There are
generalizations of this fact to some other situations (see, e.g., a survey in \cite{Euler}).
In all these cases one has no intristic explanation of the relation. The relation is
obtained by independent computation of the right and left hand sides of it in the same terms
and comparison of the obtained results.

Generalizations of relations of this sort to equivariant settings could help to understand
the general framework. This leads to the desire to define equivariant analogues of monodromy
zeta functions and of the Poincar\'e series of filtrations. These problem is not trivial
and equivariant analogues are not unique. For example, in \cite{CDG2, CDG3}, there were offered
different approaches to equivariant Poincar\'e series. In \cite{GLM1}, there was given
an equivariant version of the monodromy zeta function as a power series with the coefficients
from $K_0(\fgs)\otimes\Q$. The fact that it was defined only after tensoring by the field $\Q$
of rational numbers makes it less reasonable, in particular, to compare it with the
equivariant versions of the Poincar\'e series which were defined over integers.

One of the main ingredients of the definition of the equivariant version of the
monodromy zeta function in \cite{GLM1} was the definition of the equivariant Lefschetz
number of a transformation from \cite{Luck}. The definition from \cite{Luck} is rather natural.
Moreover, one can say that it is the only possible definition possessing some reasonable
properties. However the fact that it leads to a ``non-integer'' definition of the (monodromy)
zeta function gives a hint that this definition is not absolutely adequate to this purpose.

There is certain freedom in a definition of an equivariant 
version of the Lefschetz number of a transformation connected with the question whether
it should count the fixed points of the transformation or the fixed $G$-orbits of it.
Here we discuss the second approach, describe the corresponding 
equivariant versions of the Lefschetz number
and of the zeta function of a transformation. This zeta-function
is a power series with the coefficients from the ring $K_0(\fgs)$.

We give an A'Campo type formula for the equivariant monodromy zeta function
of a function germ in terms of a resolution. We also discuss possible orbifold versions
of the zeta function of a transformation.

In the last section  we discuss orbifold 
versions of the Lefschetz number and of the monodromy zeta function
corresponding to the two equivariant ones.

\section{Burnside ring and the equivariant Euler characteristic}\label{uno}
A {finite $G$-set} is a finite set with an action (say a left one) of the group $G$.
Isomorphism classes of irreducible $G$-sets (i.e. those which consist of exactly one orbit)
are in one-to-one correspondence with the set $\mbox{\rm consub}(G)$ of conjugacy classes
of subgroups of $G$. The Grothendieck ring $K_0(\fgs)$ of finite $G$-sets
(also called the Burnside ring of $G$) is the group generated by isomorphism classes
of finite $G$-sets with the relation $[A\coprod B]=[A]+[B]$ and with the multiplication
defined by the cartesian product.
As an abelian group
$K_0(\fgs)$  is freely generated by isomorphism classes $[G/H]$ of irreducible $G$-sets.

The Grothendieck ring $K_0(\fgs)$ has a natural pre-$\lambda$-ring structure defined by the series
$$
\sigma_X(t)=1+[X]\,t+[S^2X]\,t^2+[S^3X]\,t^3+\ldots\,,
$$
where $S^kX=X^k/S_k$ is the $k$-th symmetric power of the $G$-set $X$ with the natural $G$-action.
This pre-$\lambda$-ring structure induces a power structure over the Grothendieck ring $K_0(\fgs)$: see
\cite{Michigan}. This means that for a power series $A(t)\in 1+t\cdot K_0(\fgs)[[t]]$ and $m\in K_0(\fgs)$
there is defined a  series $\left(A(t)\right)^m \in 1+t\cdot K_0(\fgs)[[t]]$ so that
all the properties of the exponential function hold. In these notations $\sigma_X(t)=(1-t)^{-[X]}$.
The geometric description of the natural power stracture over the Grothendieck ring of quasiprojective
varieties given in \cite{Michigan} is valid for the power structure over $K_0(\fgs)$ as well.

Some examples of computation of the series $(1-t)^{-[G/H]}$ for 
$G$ being the cyclic group $\Z_6$ of order $6$ 
and the group ${\mathcal S}_3$ of permutations on three elements
can be found in \cite{GLM1} (with some misprints).
For example
\begin{eqnarray*}
(1-t)^{-[{\mathcal S}_3/\langle e\rangle]}&=&\frac{1}{1-t^6}[1]+\frac{t^3}{(1-t^3)(1-t^6)}[{\mathcal S}_3/\Z_3]\\
&+&\frac{3t^2}{(1-t^2)^2(1-t^6)}[{\mathcal S}_3/\Z_2]\\
&+&\frac{t(1+4t^2+t^3+4t^4-2t^5+3t^6+t^7)}{(1-t^2)^2(1-t^3)(1-t^6)(1-t)^2}[{\mathcal S}_3/\langle e\rangle]\,.
\end{eqnarray*}


There is a natural homomorphism from the Grothendieck ring $K_0(\fgs)$ to the ring $R(G)$
of virtual representations of the group $G$ which sends the class $[G/H]\in K_0(\fgs)$ 
to the representation $i^G_H[1_H]$ induced from the trivial one-dimensional representation $1_H$
of the subgroup $H$.
This homomorphism is a homomorphism of pre-$\lambda$-rings (\cite{knutson}).

In some places, say, in \cite{TtD, Luck, GLM1}, the  equivariant Euler
characteristic of a $G$-space is considered as an element of the Grothendieck ring $K_0(\fgs)$.
For a relatively good $G$-space $X$ (say, for a quasiprojective variety) the
equivariant Euler characteristic $\chi^G(X)\in K_0(\fgs)$ can be defined in the following way.
For a point $x\in X$, let $G_x=\{g\in G:\, g\cdot x=x \}$ be the isotropy subgroup of the point $x$.
For a conjugacy class $\gH\in \mbox{\rm consub}(G)$, let
$X^{\gH}=\{x\in X: x \mbox{ is a fixed point of a subgroup }H\in \gH\}$ and let
$X^{(\gH)}=\{x\in X: G_x\in \gH\}$ be the set of points with the isotropy subgroups from $\gH$.
(One can see that in the natural sense $X^{(\gH)}=X^{\gH}\setminus X^{>\gH}$, where
$X^{>\gH}=\bigcup\limits_{\gH'>\gH}X^{\gH'}$.) Then
\begin{equation}\label{equiEuler}
\chi^G(X)=\sum_{\gH\in {\rm consub}(G)}\frac{\chi(X^{(\gH)})\,|H|}{|G|}[G/H]=
\sum_{\gH\in {\rm consub}(G)}\chi(X^{(\gH)}/G)[G/H],
\end{equation}
where $H$ is a representative of the conjugacy class $\gH$.

\begin{remark}
Here we use the additive Euler characteristic $\chi(\cdot)$, i.e. the alternating sum of
the ranks of the cohomology groups with compact support. For a complex analytic variety this
Euler characteristic is equal to the alternating sum of the
ranks of the usual cohomology groups.
\end{remark}

Let us show that for any subgroup $H$ of $G$ the series
$(1-t)^{-[G/H]}$ represents rational function with the denominator equal to a product of
the binomials of the form $(1-t^m)$.
Let 
\begin{equation}\label{power-G-set}
(1-t)^{-[G/H]}= \sum_{\gF \in {\rm consub}(G)}{\mathcal A}_ {H, \gF}(t)[G/F]
 \end{equation}
where $F$ is a representative of the conjugacy class $\gF$.

Let $\gF$ be a conjugacy class of subgroups of $G$ and let $F$ be a representative of it.
The subgroup $F$ acts on the $G$-space $G/H$. Let $F\backslash G/H$ be the 
quotient of $G/H$ by this action and let $p:G/H \to F\backslash G/H$
be the quotient map. For $m=1,2,\ldots,$ let $Y_m$ be the set of points of $F\backslash G/H$
with $m$ preimages in $G/H$ and let ${\ell}^{\gF}_m=|Y_m|$.
(The numbers ${\ell}^{\gF}_m$ depend only on the congugacy class $\gF$.)
For an abelian $G$, ${\ell}^{\gF}_m$ is different from zero if and only if 
$m=\frac{|F|}{|F\cap H|}$ and in this case ${\ell}^{\gF}_m=|G|/|F+H|$.

For conjugacy classes $\gF$ and $\gF'$ from ${\rm consub}(G)$, let $F$ and $F'$ be their representatives,
and let $r_{\gF',\gF}$ be the number of fixed points of the group $F$
on $G/F'$. The integer $r_{\gF',\gF}$ is different from zero if and only
if $\gF'\ge \gF$ (i.e. there exist representatives $F'$ of $F$ of them such that $ F' \supset F$.
For an abelian $G$ and for $\gF'\ge \gF$, one has $r_{\gF',\gF}=|G/F'|$. 
 
\begin{lemma} For $\gF\in {\rm consub}(G)$ one has
 \begin{equation}\label{triang}
\prod_{m \geq 1} (1-t^m)^{-{\ell}^{\gF}_m}
= \sum_{\gF' \in {\rm consub}(G)}r_{\gF',\gF}\,\, {\mathcal A}_ {H, \gF'}(t).
 \end{equation}
\end{lemma}

\begin{proof}
 Let $F$ be a representative of $\gF$ and let us count fixed points of the subgroup $F$
in the left hand side and right hand side of (\ref{power-G-set}). 

For a finite set $X$ an element of $\coprod_{k\geq 0}S^kX $ can be identified 
with an integer valued function on $X$ with non-negative values. The corresponding 
element belongs to $S^kX $ if and only if the sum of all the values of the function is equal
to $k$.
An element of $\coprod_{k\geq 0}S^k [G/H] $ is fixed with respect to $F$ if and only if the corresponding 
function is invariant with respect to the $F$-action on $G/H$. Such a function can be identified 
with a function on $F\backslash G/H$. A function on $F\backslash G/H$ can be also considered 
as the direct sum of functions on the subset $Y_s$ defined above.
The generating series for the number of functions on $Y_s$ (i.e. the series 
 $\sum_{k\geq 0}\arrowvert S^k Y_s \arrowvert\, t^k$) is 
$(1-t)^{-|Y_s|}=(1-t)^{-{\ell}^{\gF}_s}$.
Each function on $Y_s$ with the sum of values equal to $k$ lifts to an $H$ invariant
 function  on $G/H$ with the sum of the values equal to $k s$. 
Therefore the generating series for $F$-invariants functions on $p^{-1}(Y_s)$ is  
$(1-t^s)^{-{\ell}^{\gF}_s}$.
The generating series for all $F$-invariants functions on $G/H$ is the product 
of those for $p^{-1}(Y_s)$. This is the left hand side of $(\ref{triang})$.

The right hand side of $(\ref{triang})$ is obviously the set of fixed points of $F$ 
on the right hand side of $(\ref{power-G-set})$.
\end{proof}

Since $r_{\gF',\gF}$ is different from zero if and only if $\gF\le \gF'$ and
$r_{\gF',\gF'}$ is different from zero, the system 
of equations $(\ref{triang})$ is a triangular one (with respect to the partial
order on the set of conjugacy classes of subgroups of $G$). Toghether with the fact 
that the denominators of the left hand side of the equations $(\ref{triang})$
are product of the binomials ot the form $(1-t^m)$ this implies the following statement.

\begin{proposition}\label{polynomials}
 For any subgroup $H$ of $G$ the series $(1-t)^{-[G/H]}$ belongs to the localization 
$K_0(\fgs)[t]_{(\{1-t^m\})}$ of the polynomial ring $K_0(\fgs)[t]$ at all the elements
of the form $(1-t^m)$, $m\geq 1$.
\end{proposition}

The natural homomorphism from the Grothendieck ring $K_0(\fgs)$ to the ring
$R(G)$ of virtual representations of the group $G$ sends the equivariant Euler
characteristic $\chi^G(X)$ to the one used in \cite{Wall}. Since this
homomorphism is, generally speaking, neither injective, no surjective, the
equivariant Euler characteristic as an element in $K_0(\fgs)$ is a somewhat
finer invariant than the one as an element of the ring $R(G)$.

\section{An alternative version of the equivariant Lefschetz number of a map}\label{EquiLn}
Let $X$ be a relatively good topological space (say, a quasiprojective complex or real variety)
with a $G$-action and let $\varphi:X\to X$ be a $G$-equivariant map. The usual (``non-equivariant'')
Lefschetz number $L(\varphi)$ counts the fixed points of $\varphi$ (or rather of its generic perturbation).
The equivariant version $L^G(\varphi)$ of the Lefschetz number from \cite{Luck} counts the
fixed points of $\varphi$ as a (finite) $G$-set. This leads to the following equation for the
equivariant Lefschetz number
\begin{equation}\label{oldL}
L^G(\varphi)=\sum\limits_{\gH\in \mbox{consub}(G)}
\frac{L(\varphi_{\vert (X^{\gH}, X^{>\gH})})\vert H\vert}{\vert G\vert} [G/H]\,,
\end{equation}
where $H$ is a representsative of the class $\gH$.
If $\varphi$ is a $G$-homeomorphism (like the monodromy transformation, see Section~\ref{Sec_acampo}), then
\begin{equation}\label{oldLhomeo}
L^G(\varphi)=\sum\limits_{\gH\in \mbox{consub}(G)}
\frac{L(\varphi_{\vert X^{(\gH)}})\vert H\vert}{\vert G\vert} [G/H]\,.
\end{equation}

Assume that one wants to count the fixed orbits of $\varphi$ (i.e. the orbits which are sent to
themselves, generally speaking, not pointwise) as finite $G$-sets. This leads to the following
definition of the equivariant Lefschetz number:
\begin{equation}\label{newL}
{\widetilde L}^G(\varphi)=\sum\limits_{\gH\in \mbox{consub}(G)}
L(\varphi_{\vert (X^{\gH}/G, X^{>\gH}/G)}) [G/H]\,.
\end{equation}
If $\varphi$ is a $G$-homeomorphism, one has
\begin{equation}\label{newLhomeo}
{\widetilde L}^G(\varphi)=\sum\limits_{\gH\in \mbox{consub}(G)}
L(\varphi_{\vert X^{(\gH)}/G}) [G/H]\,.
\end{equation}
(It is useful to compare Equations (\ref{oldL}), (\ref{oldLhomeo}) and (\ref{newL}),
(\ref{newLhomeo}) with the two parts of the equation (\ref{equiEuler}).)

\begin{example}
For some simplicity, let the group $G$ be abelian, let $X=(G/H)\times \Z_k=\{0,1, \ldots, k-1\}$, $k>0$,
with the natural action of the group $G$ on the first factor, and let the map
$\varphi:X\to X$ be defined by
$$
\varphi(a, i)= \begin{cases} (a, i+1)  & \mbox{ for } 0\le i< k-1\,,\\
(ga, 0)  & \mbox{ for } i=k-1\,. \end{cases}
$$
If $k>1$, then $L^G(\varphi)={\widetilde L}^G(\varphi)=0$ since $\varphi$ has neither fixed points,
no fixed orbits. The smallest $i$ for which ${\widetilde L}^G(\varphi^i)\ne 0$
is $i=k$. In this case all the $G$-orbits in $X$ are fixed by $\varphi^k$
and therefore ${\widetilde L}^G(\varphi^k)=k[G/H]$. On the other hand, if
$g\notin H$, the map $\varphi^k$ has no fixed points and thus
$L^G(\varphi^k)=0$. The smallest $i$ for which $L^G(\varphi^i)\ne 0$
is $i=\ell k$, where $\ell$ is the order of the element $g$ in the
group $G/H$. In this case all the points of $X$ are fixed by
$\varphi^{\ell k}$
and therefore ${\widetilde L}^G(\varphi^{\ell k})=k[G/H]$.
\end{example}

Just in the same way as in \cite{Luck} one can formulate the equivariant version
of the Lefschetz fixed point theorem for ${\widetilde L}^G(\varphi)$ (an analogue
of Theorem~2.1 in \cite{Luck}).

\section{The zeta function of a transformation}\label{Sec_zeta}
Let  $\varphi:X\to X$ be as above. The usual (non-equivariant) zeta function of $\varphi$
is defined in terms of the action of $\varphi$ in the (co)homology groups of $X$ (in a way somewhat
similar to the definition of the Lefschetz number).  This definition is not convinient for a direct generalizarion
to the equivariant case. It is more convinient to use the  definition of the zeta function
of the transformation $\varphi$ in terms of the   Lefschetz numbers of the iterates  of $\varphi$. One defines
integers  $s_i$, $i=1, 2\ldots,$ recursively by the equation
\begin{equation}\label{lefsc}
L(\varphi^m)=\sum_{i|m}s_i\,\, .
 \end{equation}

The number $s_m$ counts the points $x\in X$
with the $\varphi$-order equal to $m$ (i.e.
$\varphi^m(x)=x$,  $\varphi^{i}(x)\ne x$ for $0<i<m$). Together with each such point all its images
under the iterates of $\varphi$ (there  are exactly $m$ different ones) are of this sort. Therefore
$s_m$ is divisible by $m$. One defines the zeta function $\zeta_\varphi(t)$ to be
\begin{equation}\label{zetalef}
\zeta_\varphi(t):=\prod_{m\geq 1} (1-t^m)^{-{s_m}/{m}}.
\end{equation}

\begin{remark}
 There are two traditions to define the zeta function of a transformation. The other one does not contain
the minus sign in the exponent and therefore is the inverse to this one. Here we follow the definition from \cite{AC}.
\end{remark}

In the equivariant version, let $s_m^{{{G}}}(\varphi)$ and ${\widetilde s}_m^{{{G}}}(\varphi)$
be defined through $
L^G(\varphi^i)$ and $
{\widetilde L}^G(\varphi^i)$  respectively by the analogues of the equation  (\ref{lefsc})
\begin{equation}\label{G-lefsc}
L^G(\varphi^m)=\sum_{i|m}s_i^G(\varphi),\,\,\,  {\widetilde L}^G(\varphi^m)=\sum_{i|m}{\widetilde s}_i^G(\varphi).
 \end{equation}
The elements $s_m^G(\varphi)$ and ${\widetilde s}_m^G(\varphi)$ count the points in $X$ the $\varphi$-order
of which is equal to $m$ in
$X$ and in $X/G$ respectively.

\begin{example}
 In the Example from Section~\ref{EquiLn}
$${\widetilde s}_i^G(\varphi)=0\,\,\mbox{ for }\,\, i<k \,\,\mbox{ and  }\,\,{\widetilde s}_k^G(\varphi)=
{\widetilde L}^G(\varphi)=k [G/H],$$
$${s}_i^G(\varphi)=0\,\,\mbox{ for }\,\, i<\ell k \,\,\mbox{ and  }\,\,{ s}_{\ell k}^G(\varphi)=
{ L}^G(\varphi)=k [G/H].$$
\end{example}
One can see that in this case ${\widetilde s}_k^G(\varphi)$ is divisible by $k$,
but ${s}_{\ell k}^G(\varphi)$ is not divisible by $\ell k$.
This is a general feature. One can easily prove the following proposition.

\begin{proposition}
 The element ${\widetilde s}_m^G(\varphi)$ in $K_0(\fgs)$ is divisible by $m$.
\end{proposition}

This permits to give the following definition:

\begin{definition}
The {equivariant zeta function} of a $G$-equivariant map $\varphi:X\to X$ is the series
${\widetilde \zeta}_\varphi^G(t)\in 1+t\cdot K_0(\fgs)[[t]]$ defined by
\begin{equation}\label{eq2}
{\widetilde \zeta}_\varphi^G(t)=\prod_{m\geq 1} (1-t^m)^{-{{\widetilde s}_m^{{{G}}}}/{m}},
\end{equation}
where the virtual finite $G$-sets ${\widetilde s}_m^G\in K_0(\fgs)$ are defined as above.
\end{definition}

\begin{remark}
The definition given in \cite{GLM1} used the elements ${s}_m^G$ and made sense only after tensoring by the field $\Q$.
\end{remark}

Applying the natural homomorphism from the Grothendieck ring $K_0(\fgs)$
to the ring $R(G)$ of representations of the group $G$ (see Section \ref{uno})
one gets a reduced version of the zeta function as an element of $1+t\cdot R(G)[[t]]$.
The equivariant Poincar\'e series of filtrations defined in \cite{CDG2} also belongs to the set
$1+t\cdot R(G)[[t]]$. (In fact the latter one even belongs to the set  $1+t\cdot R_1(G)[[t]]$,
where $R_1(G)$ is the subring of $R(G)$ generated by the one-dimensional representations.
However, for an abelian group $G$, where the notion of the equivariant Poincar\'e series from \cite{CDG2}
really makes sense, these sets coincide.)

The properties of the equivariant Lefschetz numbers and the example above imply  the following propositions.

\begin{proposition} \label{stz1}
 Let $\varphi:X \to X$ be such that $\varphi(Y)\subset Y$,   $\varphi(X\setminus Y)\subset X\setminus Y$
for a $G$-subset $Y\subset X$. Then
$$
{\widetilde \zeta}_\varphi^G(t)=
{\widetilde \zeta}_{\varphi_{\vert Y}}^G(t)\cdot {\widetilde \zeta}_{\varphi_{\vert {X\setminus Y}}}^G(t)\,.
$$
\end{proposition}

\begin{proposition} \label{stz2}
 Let $X=X^{(\gH)}$ for a conjucacy class $\gH\in \mbox{consub}(G)$ and let   $\varphi:X \to X$
be a $G$-equivariant map. Then
$$
{\widetilde \zeta}_\varphi^G(t)=\left({\zeta}_{\varphi_{\vert X/G}}(t)\right)^{[G/H]}
$$
for a representative $H$ of the class $\gH$.
\end{proposition}

Assume that a $G$-equivariant map $\varphi:X \to X$ preserves the subspaces 
$X^{(\gH)}$, i.e. $\varphi(X^{(\gH)})\subset X^{(\gH)}$. In this case 
$\varphi_{\vert X/G}$ preserves the subspaces $X^{(\gH)}/G$. 
Propositions (\ref{stz1}) and  (\ref{stz2}) imply the following statement.

\begin{proposition} \label{stz3}
Let $\varphi:X \to X$ be a $G$-equivariant map that preserves the subspaces 
$X^{(\gH)}$ for any conjucacy class $\gH\in \mbox{consub}(G)$. Then
$$
{\widetilde \zeta}_\varphi^G(t)=
\prod_{\gH\in \mbox{consub}(G)}
\left({\zeta}_{\varphi_{\vert X^{(\gH)}/G}}(t)\right)^{[G/H]}.
$$
\end{proposition}

This statement can be applied to the monodromy transformation below.

\section{The A'Campo type formula for the equivariant monodromy zeta function}\label{Sec_acampo}
Let $(V,0)$ be a germ of a purely $n$-dimensional complex analytic variety with an action of the group $G$
and let $f:(V,0)\to (\C,0)$ be the germ of a $G$-invariant analytic function such that
$\mbox{\rm Sing} \,V\subset f^{-1}\{0\}$.
Let $M_f$ be the Milnor fibre of the germ $f$ at the origin:
$M_f=\{x\in V: f(x)=\varepsilon, \Vert x \Vert\leq \delta\,\}$ with
$0<\vert \varepsilon \vert \ll \delta $ small enough (for this definition we assume $(V,0)$ to be embedded
in the affine space $(\C^N,0)$). The group $G$ acts on the Milnor fibre $M_f$.
The classical monodromy transformation $h=h_f:M_f\to M_f$
is a $G$-equivariant map corresponding to the loop $\varepsilon(\tau)=
\varepsilon\cdot\exp{(2\pi i \tau)}$ in $\C$ around the origin.

Let $\pi:(X,\DD)\to (V,0)$ be a $G$-equivariant resolution of the germ $f$, i.e. a proper
$G$-equivariant map from an $n$-dimensional $G$-manifold $X$ to $V$ such that $\pi$ is an isomorphism outside
the zero level set  $f^{-1}\{0\}$ of the function $f$ and, in a neighbourhood of any point $p$ of the total transform
$E_0:=\pi^{-1}(f^{-1}\{0\})$ of the zero-level set $f^{-1}\{0\}$ of the function $f$, there exists a local
system of coordinates $z_1,\ldots,z_n$
(centred at the point $p$) such that $f\circ \pi(z_1,\ldots,z_n)=z_1^{m_1}z_2^{m_2}\cdots z_n^{m_n}$,
with non-negative integers $m_i$, $i=1,\ldots,n$. (This implies that the total transform $E_0$ of
the zero level set $f^{-1}\{0\}$ is a normal crossing divisor on $X$.)
Moreover, we assume that, for each point
$p\in E_0$, the irreducible components of $E_0$ at the point $p$ are invariant with respect to the isotropy
group $G_p$ of the point $p.$
(This can be achieved, if necessary, by additional blow-ups  of the intersections of the components of the
exceptional divisor.)

Let $S_m$, $m\geq 1$, be the set of points $p$ of the exceptional divisor $\DD:=\pi^{-1}\{0\}$ such that,
in a neighbourhood of the point $p$, one has $f\circ \pi(z_1,\ldots,z_n)=z_1^{m}$.
For $p\in S_m$, let $G_p$ be the isotropy group of the point $p$: $G_p=\{g\in G:\,gp=p\}$.
The group $G_p$ acts on the smooth germ $(X,p)$ preserving the exceptional divisor $\DD$ locally
given by $z_1=0$. This implies that, in a neighbourhood of the point $p$, one can suppose $G_p$ to act by
linear transformations in coordinates $z_1,\ldots,z_n$ preserving ``the normal slice''
$z_2=\ldots=z_n=0$. This way one gets a linear representation of the group $G_p$ on this
normal slice.
Let ${\widehat G}_p$ be the kernel of this representation. One can see that $G_p/{\widehat G}_p$
is a cyclic group the order of which divides $m$.

 Let $S_{m,H,{\widehat H}}$, $({\widehat H}\subseteq H)$, be the set of points $p\in S_m$
such that the pair $(H,{\widehat H})$ is conjugate to the pair $ (G_p,{\widehat G}_p)$, i.e.
for an element $g\in G$ one has  $G_p=gHg^{-1}$ and ${\widehat G}_p=g{\widehat H}g^{-1}$.

\begin{theorem*}\label{acampo}
\begin{eqnarray}\label{eq-acampo}
{\widetilde \zeta}^G_f(t)&=&\prod_{m\geq 1,(\gH,{\widehat{\gH}})}(1-t^{m\frac{\vert \widehat H \vert}{\vert  H \vert}})
^{-\frac{|{ H}|\chi(S_{m,H,{\widehat H}})}{|G|}[G/{\widehat H}]} \nonumber \\
&=&\prod_{m\geq 1,(\gH,{\widehat{\gH}})}(1-t^{m\frac{\vert \widehat H \vert}{\vert  H \vert}})^
{-{\chi(S_{m,H,{\widehat H}}/G)}[G/{\widehat H}]}\,,\label{eq4}
\end{eqnarray}
where the product is over all conjugacy classes $(\gH,{\widehat{\gH}})$ of pairs of subgroups of the group $G$
(such that $H/ {\widehat H}$ is cyclic),
the pair $(H,{\widehat H})$ is a representative of the conjugacy class $(\gH,{\widehat{\gH}})$.
\end{theorem*}

\begin{proof}
Since the resolution $\pi$ is an isomorphism ouside of $f^{-1}(0)$,
the Milnor fibre $M_f$ can be identified with its preimage $\pi^{-1}(M_f)\subset X$.
Just like in the non-equivariant case (see \cite{AC},\cite{clemens}) one can construct
a $G$-equivariant retraction of a neighbourhood of the total transform $E_0$ of the zero
level set $f^{-1}\{0\}$ to $E_0$ itself such that, outside of a neighbourhood in $E_0$
of all intersections of at least two irreducible components of $E_0$, i.e. where in
local coordinates  $f\circ \pi(z_1,\ldots,z_n)=z_1^{m},\, m\geq 1$, the retraction sends
a point $(z_1,z_2,\ldots,z_n)$ to $(0,z_2,\ldots,z_n)$.
Moreover, one can construct a monodromy transformation $h=h_f$ which commutes with the retraction
and, in a neighbourhood of a point where $f\circ \pi(z_1,\ldots,z_n)=z_1^{m}$, it sends
a point
$(z_1,z_2,\ldots,z_n)\in M_f$ to $(\exp(\frac{2\pi i}{m})z_1,z_2,\ldots,z_n)\in M_f$, $M_f$ being the Milnor fibre.

From Statement~\ref{stz1}, it follows that the equivariant zeta function
${\widetilde \zeta}_f^G(t):={\widetilde \zeta}_{h_f}^G(t)$ is equal to the product
of the equivariant zeta functions for the monodromy transformation on the preimages of
the strata $S_{m,H,{\widehat H}}$ and on  neighbourhoods of the
intersections of the irreducible components of $E_0$.

One can show that the equivariant zeta function
of the monodromy transformation $h$ on a neighbourhood of an intersection of the
irreducible components of $E_0$ is equal to one.
This follows from the fact that the Milnor fibre $M_f$ in a neighbourhood of such an intersection
can be fibred by circles and the monodromy transformation can be supposed to preserve these fibration.

To compute the equivariant zeta function  of the monodromy transformation on the preimage of the stratum
$S_{m,H,{\widehat H}}$ it is useful to have in mind the following property of the usual zeta funtion
of a transformation.
Let $\varphi:Y\to Y$ be such that $\varphi^i$ has no fixed points for $0<i<k$, $\varphi^k=id$. Then
${\widetilde \zeta}_{\varphi}(t)=(1-t^k)^{-\chi(Y)/k}$.
Now the  mentioned above equivariant zeta function
is given by Statement~\ref{stz2} and is equal to
\begin{equation*}
(1-t^{m\frac{\vert \widehat H \vert}{\vert  H \vert}})
^{-\frac{|{ H}|\chi(S_{m,H,{\widehat H}})}{|G|}[G/{\widehat H}]}
=(1-t^{m\frac{\vert \widehat H \vert}{\vert  H \vert}})^
{-{\chi(S_{m,H,{\widehat H}}/G)}[G/{\widehat H}]}\,.
\end{equation*}
This implies  (\ref{eq4}).
\end{proof}

Equation (\ref{eq-acampo}) with Proposition  (\ref{polynomials})
implies the followinf statement.
 
\begin{proposition}
 The equivariant monodromy zeta function ${\widetilde \zeta}_f^G(t)$ is a rational function
with coefficients in the Grothendieck ring $K_0(\fgs)$ of finite $G$-sets.
 \end{proposition}

\begin{examples}
\begin{enumerate}
\item  Assume that the classical monodromy transformation $h=h_f:M_f\to M_f$ is an element of the group $G$.
This happens, in particular, when $f$ is a quasi-homogeneous function and $G$ is its symmetry group.
In this case the action of $h$ on $M_f/G$ is trivial and therefore ${\widetilde \zeta}_f^G(t)=(1-t)^{-\chi^G(M_f)}$.
Thus in this case the equivariant monodromy zeta function is determined by the equivariant Euler characteristic
$\chi^G(M_f)$ of the Milnor fibre $M_f$. This corresponds to the idea that in this situation
the equivariant Euler characteristic  $ \chi^G(M_f)$ can be considered as an equivariant analogue
of the monodromy zeta function: \cite{EG}.

\item Let  $f:(\C^3,0)\to (\C,0)$ be defined by  $f(x,y,z)=x^m+y^m+z^m$.
Consider the natural action of the group ${\mathcal S}_3$ of permutations on three elements on $\C^3$
by permutations of the cooordinates. The function $f$ is ${\mathcal S}_3$-invariant.
Blowing-up the origin, one gets a resolution of the function $f$.
Let us assume that $m=6k$. (In the other cases some of the strata $S_{m,H,{\widehat H}}$
below can be empty. These cases can be treated in the same way.)
One has the following  strata $S_{m,H,{\widehat H}}$:
\newline a) $S_{m,{\mathcal S}_3,{\mathcal S}_3}$ consists of one point $P=(1:1:1)\in \C\P^2=\pi^{-1}(0)$.
\newline b) $S_{m,\Z_2,\Z_2}$ consists of three lines $\{x=y\}$, $\{x=z\}$ and $\{y=z\}$ (passing trough $P$)
without their intersections with the strict transform  of the surface $\{f=0\}$, i.e. with the curve
$C=\{x^m+y^m+z^m=0\} \subset \C\P^2$. One has $\chi(S_{m,\Z_2,\Z_2}/{\mathcal S}_3)=1-6k.$
\newline c) $S_{m,\Z_2,\langle e \rangle}$ consists of three points $(1:-1:0)$, $(1:0:-1)$ and $(0:1:-1)$.
(For $6|m$ these points do not lie on the curve $C$.) One has $\chi(S_{m,\Z_2,\langle e \rangle}/{\mathcal S}_3)=1$.
\newline d) $S_{m,\Z_3,\langle e \rangle}$ consists of two points $(1:\sigma:\sigma^2)$ and $(1:\sigma^2:\sigma)$,
where $\sigma=\exp(2\pi i/3).$ (For $6|m$ these points do not lie on  $C$.) One has
$\chi(S_{m,\Z_3,\langle e \rangle}/{\mathcal S}_3)=1$.
\newline e) $S_{m,\langle e \rangle,\langle e \rangle}$ is the complement in $\C\P^2$ to the curve $C$
and to all the strata above. One has $\chi(S_{m,\langle e \rangle,\langle e \rangle}/{\mathcal S}_3)=6k^2-1$.

 Now A'Campo formula  (\ref{eq4}) gives
 \begin{eqnarray*}
{\widetilde \zeta}^{{\mathcal S}_3}_f(t)&=&(1-t^{6k})^{-1}\cdot (1-t^{6k})^{(6k-1)[{\mathcal S}_3/\Z_2]}
\cdot (1-t^{3k})^{-[{\mathcal S}_3/\langle e \rangle]}
\cdot \\
& \ & (1-t^{2k})^{-[{\mathcal S}_3/\langle e \rangle]}
\cdot (1-t^{6k})^{(1-6k^2)[{\mathcal S}_3/\langle e \rangle]}\,.
    \end{eqnarray*}
\end{enumerate}

\end{examples}

\section{On orbifold versions of the monodromy zeta functions}\label{Sec_0rb}
For a $G$-variety $X$, its \emph{orbifold Euler characteristic} 
$\chi^{orb}(X, G)\in \Z$  is defined, e.g., 
in \cite{AS} or \cite{HH}. There is a natural homomorphism of abelian groups $\Phi:K_0(\fgs)\to  \Z$ which
sends the equivariant Euler characteristic $\chi^G(X)\in K_0(\fgs)$ to 
the orbifold Euler characteristic $\chi^{orb}(X, G)$. The homomorphism $\Phi$ sends 
the generator $[G/H]$ of  $K_0(\fgs)$ to 
$\chi^{orb}(G/H,G)$. For an abelian $G$, $\Phi([G/H])=\vert H \vert$ and $\Phi$ 
is a ring homomorphism, but this is  not the case in general. 

The Lefschetz number is a sort of generalization of the Euler characteristic: 
the  Euler characteristic is the  Lefschetz number of the identity map.
Therefore an equivariant version of the  Lefschetz number leads to the natural orbifold 
versions of it: the images of the  Lefschetz number by the homomorphism $\Phi$.
The two versions $L^G(\varphi)$ and ${\widetilde L}^G(\varphi)$ of the
equivariant Lefschetz number (see (\ref{oldL}) and (\ref{newL}))
give two versions 
$$L^{orb}(\varphi)=\Phi( L^G(\varphi))
\,\,\mbox{ and  }\,\,{\widetilde L}^{orb}(\varphi)=\Phi( {\widetilde L}^G(\varphi))$$
of orbifold Lefschetz numbers. 

The usual definition of the zeta function of a transformation (e.g. equations 
(\ref{lefsc}), (\ref{zetalef}), (\ref{G-lefsc}) and (\ref{eq2})) 
gives two orbifold versions of the zeta function 
of a $G$-equivariant transformation $\varphi:X\to X$:
\begin{equation}\label{zeta-orb}
{\zeta}_\varphi^{orb}(t)=
\prod_{m\geq 1} (1-t^m)^{-{{s}_m^{{orb}}}/{m}},
\,\,\mbox{ and  }\,\,
{\widetilde \zeta}_\varphi^{orb}(t)=
\prod_{m\geq 1} (1-t^m)^{-{{\widetilde s}_m^{{orb}}}/{m}},
\end{equation}
where $ L^{orb}(\varphi^m)=\sum_{i|m}s^{orb}_i(\varphi)$
and  $ {\widetilde L}^{orb}(\varphi^m)=\sum_{i|m}{\widetilde s}^{orb}_i(\varphi)$.

The exponents $-{{\widetilde s}_m^{{orb}}}/{m}$ are integers and therefore 
the orbifold monodromy zeta function ${\widetilde \zeta}_\varphi^{orb}(t)$ 
is a rational function in $t$. The exponents $-{{s}_m^{{orb}}}/{m}$ are 
in general rational numbers. 
 
For intance, for $f$ from Example 2 in Section~4 one has
\begin{equation*}
 {\widetilde \zeta}_f^{orb}(t)
=(1-t^{6k})^{-1+2(6k-1)+(1-6k^2)}
\cdot (1-t^{3k})^{-1}
\cdot  (1-t^{2k})^{-1}.
\end{equation*}
This follows from the fact that, for $G={\mathcal S}_3$ and for a subgroup $H$ of ${\mathcal S}_3$, 
$\chi^{orb}({\mathcal S}_3/H,{\mathcal S}_3)=\vert H \vert$.

\end{document}